\documentclass[a4paper,oneside,10pt]{article}%
\usepackage{amsfonts}
\usepackage{amsmath,amssymb}
\usepackage{amsmath}
\usepackage{amssymb}
\usepackage{graphicx}
\usepackage{hyperref}%
\providecommand{\U}[1]{\protect\rule{.1in}{.1in}}
\newtheorem{theorem}{Theorem}[section]

\newtheorem{definition}[theorem]{Definition}

\newtheorem{lemma}[theorem]{Lemma}
\newtheorem{proposition}[theorem]{Proposition}
\newtheorem{remark}[theorem]{Remark}
\newenvironment{proof}[1][Proof]{\noindent \textbf{#1.} }{\  $\Box$}
\numberwithin{equation}{section}
\begin{document}

\title{Girsanov theorem for $G$-Brownian motion: the degenerate case}
\author{Guomin Liu \thanks{School of Mathematical Sciences, Fudan University, Shanghai 200433, People's Republic of China}
\\[7pt] gmliusdu@163.com }
\date{}
\maketitle

\textbf{Abstract}. In this paper, we prove the Girsanov theorem for
$G$-Brownian motion without the non-degenerate condition. The proof is based
on the perturbation method in the nonlinear setting by constructing a product
space of the $G$-expectation space and a linear space that contains a standard
Brownian motion. The estimates for exponential martingales of $G$-Brownian
motion are important for our arguments.


\textbf{AMS 2010 subject classifications:} 60H10, 60H30 \addcontentsline{toc}{section}{\hspace*{1.8em}Abstract}

\section{Introduction}

Motivated by financial problems with model uncertainty, Peng
\cite{peng2005,peng2008,Peng 1} systematically introduced the nonlinear
$G$-expectation theory. Under the $G$-expectation framework, the $G$-Brownian
motion and related It\^{o}'s stochastic calculus were constructed. Moreover,
the existence and uniqueness theorem of (forward and backward) stochastic
differential equations driven by $G$-Brownian motion were obtained in Gao
\cite{Gao}, Peng \cite{Peng 1} and Hu, Ji, Peng and Song \cite{HJPS}.

$G$-Brownian motion $B=(B_{t})_{t\geq0}$ is a continuous process with
independent and stationary increments under $G$-expectation $\hat{\mathbb{E}}%
$. It is characterized by a function $G(A)=\hat{\mathbb{E}}[\langle
AB_{1},B_{1}\rangle],$ for $A\in\mathbb{S}(d)$, where $\langle\cdot
,\cdot\rangle$ is the inner product for vectors and $\mathbb{S}(d)$ is the
sets of symmetric $d\times d$ matrices. We say that the function $G$ (or
$G$-Brownian motion $B$) is non-degenerate if there exist a constant
$\underline{\sigma}^{2}>0$ such that
\begin{equation}
G(A)-G(A^{\prime})\geq\frac{1}{2}\underline{\sigma}^{2}\text{tr}[A-A^{\prime
}],\text{ for }A\geq A^{\prime}. \label{Myeq1.1}%
\end{equation}
Under this non-degenerate condition, Osuka \cite{Osuka} and Xu, Shang and
Zhang \cite{XSZ} proved the Girsanov theorem for $G$-Brownian motion. Their
arguments used the so-called PDE method which usually applies Taylor's
expansion or It\^{o}'s formula to the solutions of $G$-heat equations that
corresponding to $G$-Brownian motion. So this method relies heavily on the
non-degenerate condition since the later guarantees the regularity of the solutions.

The aim of this paper is to generalize the Girsanov theorem to the case that
the non-degenerate condition (\ref{Myeq1.1}) for $B$ may not hold. Using the
product space theory in the nonlinear expectation setting, we obtain a
non-degenerate $G$-Brownian motion perturbation by adding a small linear
Brownian motion term to $G$-Brownian motion $B$. Then the Girsanov theorem for
$G$-Brownian motion under the non-degenerate condition applies. To get the
results for $B$, we consider a limit procedure, and the main difficulty is
that the dominated convergence theorem does not hold under the nonlinear
framework. We overcome this problem by utilizing the exponential martingale
property of $G$-Brownian motion and proving some useful estimates.

The paper is organized as follows. In Section 2, we recall some basic notions
and results of $G$-expectation, $G$-Brownian motion and Girsanov theorem for
$G$-Brownian motion in the non-degenerate case. In Section 3, we give the main
results on Girsanov theorem for possibly degenerate $G$-Brownian motion.

\section{Preliminaries}

In this section, we review some basic notions and results of $G$-expectation,
$G$-Brownian motion and the corresponding Girsanov theorem. More relevant
details can be found in \cite{Osuka,peng2005,peng2008,Peng 1,XSZ}.

\subsection{$G$-expectation space}

Let $\Omega$ be a given nonempty set and $\mathcal{H}$ be a linear space of
real-valued functions on $\Omega$ such that if $X_{1}$,$\dots$,$X_{d}%
\in\mathcal{H}$, then $\varphi(X_{1},X_{2},\dots,X_{d})\in\mathcal{H}$ for
each $\varphi\in C_{b.Lip}(\mathbb{R}^{d})$, where $C_{b.Lip}(\mathbb{R}^{d})$
is the space of bounded, Lipschitz functions on $\mathbb{R}^{d}$.
$\mathcal{H}$ is considered as the space of random variables.

\begin{definition}
A sublinear expectation $\hat{\mathbb{E}}$ on $\mathcal{H}$ is a functional
$\mathbb{\hat{E}}:\mathcal{H}\rightarrow\mathbb{R}$ satisfying the following
properties: for each $X,Y\in\mathcal{H}$,

\begin{description}
\item[\textrm{(i)}] {Monotonicity:}\quad$\mathbb{\hat{E}}[X]\geq
\mathbb{\hat{E}}[Y]\ \ \text{if}\ X\geq Y$;

\item[\textrm{(ii)}] {Constant preserving:}\quad$\mathbb{\hat{E}%
}[c]=c\ \ \ \text{for}\ c\in\mathbb{R}$;

\item[\textrm{(iii)}] {Sub-additivity:}\quad$\mathbb{\hat{E}}[X+Y]\leq
\mathbb{\hat{E}}[X]+\mathbb{\hat{E}}[Y]$;

\item[\textrm{(iv)}] {Positive homogeneity:}\quad$\mathbb{\hat{E}}[\lambda
X]=\lambda\mathbb{\hat{E}}[X]\ \ \ \text{for}\ \lambda\geq0$.
\end{description}

The triple $(\Omega,\mathcal{H},\mathbb{\hat{E}})$ is called a sublinear
expectation space.
\end{definition}

\begin{definition}
Let $(\Omega,\mathcal{H},\mathbb{\hat{E}})$ be a sublinear expectation space.
A $d$-dimensional random vector $Y$ is said to be independent from another
$m$-dimensional random vector $X$ under $\mathbb{\hat{E}}[\cdot]$ if, for each
test function $\varphi\in C_{b.Lip}(\mathbb{R}^{m+d})$, we have
\[
\mathbb{\hat{E}}[\varphi(X,Y)]=\mathbb{\hat{E}}[\mathbb{\hat{E}}
[\varphi(x,Y)]_{x=X}].
\]

\end{definition}

A family of $d$-dimensional random vectors $(X_{t})_{t\geq0}$ on the same
sublinear expectation space $(\Omega,\mathcal{H},\mathbb{\hat{E}})$ is called
a $d$-dimensional stochastic process.

\begin{definition}
Two $d$-dimensional processes $(X_{t})_{t\geq0}$ and $(Y_{t})_{t\geq0}$
defined respectively on sublinear expectation spaces $(\Omega_{1}
,\mathcal{H}_{1},\mathbb{\hat{E}}_{1})$ and $(\Omega_{2},\mathcal{H}
_{2},\mathbb{\hat{E}}_{2})$ are called identically distributed, denoted by
$(X_{t})_{t\geq0}\overset{d}{=}(Y_{t})_{t\geq0}$, if for each $n\in\mathbb{N}
$, $0\leq t_{1}<\cdots<t_{n}$, $(X_{t_{1}},\ldots,X_{t_{n}})\overset{d}{=}%
(Y_{t_{1}},\ldots,Y_{t_{n}})$, i.e.,
\[
\mathbb{\hat{E}}_{1}[\varphi(X_{t_{1}},\ldots,X_{t_{n}})]=\mathbb{\hat{E}}
_{2}[\varphi(Y_{t_{1}},\ldots,Y_{t_{n}})]\text{ for each }\varphi\in
C_{b.Lip}(\mathbb{R}^{n\times d}).
\]

\end{definition}

\begin{definition}
A $d$-dimensional process $(X_{t})_{t\geq0}$ on a sublinear expectation space
$(\Omega,\mathcal{H},\mathbb{\hat{E}})$ is said to have independent increments
if, for each $0\leq t_{1}<\cdots<t_{n}$, $X_{t_{n}}-X_{t_{n-1}}$ is
independent from $(X_{t_{1}},\ldots,X_{t_{n-1}})$. A $d$-dimensional process
$(X_{t})_{t\geq0}$ is said to have stationary increments if, for each $t$,
$s\geq0$, $X_{t+s}-X_{s}\overset{d}{=}X_{t}$.
\end{definition}

\begin{definition}
A $d$-dimensional process $(B_{t})_{t\geq0}$ on $(\Omega,\mathcal{H}%
,\hat{\mathbb{E}})$ is called a $G$-Brownian motion if the following
properties are satisfied:

\begin{description}
\item[(1)] $B_{0}=0$;

\item[(2)] It is a process with stationary and independent increments;

\item[(3)] For each $t\geq0$, $\mathbb{\hat{E}}[\varphi(B_{t})]=u^{\varphi
}(t,0)$ for each $\varphi\in C_{b.Lip}(\mathbb{R}^{d})$, where $u^{\varphi}$
is the viscosity solution of the following $G$-heat equation:
\[
\left\{
\begin{array}
[c]{l}%
\partial_{t}u(t,x)-G(D_{xx}^{2}u(t,x))=0,\\
u(0,x)=\varphi(x).
\end{array}
\right.
\]
Here $G(A):=\hat{\mathbb{E}}[\langle AB_{1},B_{1}\rangle],$ for $A\in
\mathbb{S}(d)$, where $\langle\cdot,\cdot\rangle$ is the inner product for
vectors and $\mathbb{S}(d)$ is the sets of symmetric $d\times d$ matrices.
\end{description}
\end{definition}

\begin{remark}
\upshape{If $G(A)=\frac{1}{2}\text{tr}[A]$, for $A\in
	\mathbb{S}(d)$,  then $B$ is a  standard Brownian motion.}
\end{remark}

Now we recall the construction of $G$-Brownian motion on the path space. We
denote by $\Omega:=C_{0}([0,\infty);\mathbb{R}^{d})$ the space of all
$\mathbb{R}^{d}$-valued continuous paths $(\omega_{t})_{t\geq0}$ started from
the origin and equipped with the distance
\[
\rho(\omega^{1}, \omega^{2}):=\sum^{\infty}_{N=1} 2^{-N} [(\max_{t\in[0,N]} |
\omega^{1}_{t}-\omega^{2}_{t}|) \wedge1].
\]

Let $B_{t}(\omega):=\omega_{t}$ for $\omega\in\Omega$, $t\geq0$ be the
canonical process. We set
\[
L_{ip}(\Omega_{T}):=\left\{  \varphi(B_{t_{1}},B_{t_{2}}-B_{t_{1}}%
\cdots,B_{t_{n}}-B_{t_{n-1}}):n\in\mathbb{N},0\leq t_{1}<t_{2}\cdots<t_{n}\leq
T,\varphi\in C_{b.Lip}(\mathbb{R}^{d\times n})\right\}
\]
as well as
\begin{equation}
L_{ip}(\Omega):=\bigcup_{m=1}^{\infty}L_{ip}(\Omega_{m}).
\label{9237257894334}%
\end{equation}

Let $G:\mathbb{S}(d)\rightarrow\mathbb{R}$ be a monotonic and sublinear
function. We define the $G$-expectation $\mathbb{\hat{E}} :L_{ip}%
(\Omega)\rightarrow\mathbb{R}$ by two steps.

Step 1. For $X=\varphi(B_{t+s}-B_{s})$ with $t$, $s\geq0$ and $\varphi\in
C_{b.Lip}(\mathbb{R}^{d})$, we define%
\[
\mathbb{\hat{E}}[X]=u(t,0),
\]
where $u$ is the solution of the following ${G}$-heat equation:%
\[
\partial_{t}u-{G}(D_{xx}^{2}u)=0,\ u(0,x)=\varphi(x).
\]

Step 2. For $X=\varphi(B_{t_{1}}-B_{t_{0}},B_{t_{2}}-B_{t_{1}},\cdots
,B_{t_{n}}-B_{t_{n-1}})$ with $0\leq t_{0}<\cdots<t_{n}$ and $\varphi\in
C_{b.Lip}(\mathbb{R}^{d\times n})$, we define%
\[
\mathbb{\hat{E}}[X]=\varphi_{n},
\]
where $\varphi_{n}$ is obtained via the following procedure:%
\[%
\begin{array}
[c]{rcl}%
\varphi_{1}(x_{1},\cdots,x_{n-1}) & = & \mathbb{\hat{E}}[\varphi(x_{1}%
,\cdots,x_{n-1},B_{t_{n}}-B_{t_{n-1}})],\\
\varphi_{2}(x_{1},\cdots,x_{n-2}) & = & \mathbb{\hat{E}}[\varphi_{1}%
(x_{1},\cdots,x_{n-2},B_{t_{n-1}}-B_{t_{n-2}})],\\
& \vdots & \\
\varphi_{n} & = & \mathbb{\hat{E}}[\varphi_{n-1}(B_{t_{1}}-B_{t_{0}})].
\end{array}
\]
The corresponding conditional expectation $\mathbb{\hat{E}}_{t}$ of $X$ with
$t=t_{i}$ is defined by%
\[
\mathbb{\hat{E}}_{t_{i}}[X]=\varphi_{n-i}(B_{t_{1}}-B_{t_{0}},\cdots,B_{t_{i}%
}-B_{t_{i-1}}).
\]

For each $p\geq1$, we denote by $L_{G}^{p}(\Omega_{t})$ the completion of
$L_{ip}(\Omega_{t})$ under the norm $||X||_{p}:=(\hat{\mathbb{E}}%
[|X|^{p}])^{1/p}$. The $G$-expectation $\hat{\mathbb{E}}[\cdot]$ and
conditional $G$-expectation $\hat{\mathbb{E}}_{t}[\cdot]$ can be extended
continuously to $L_{G}^{1}(\Omega)$ and $(\Omega,L_{G}^{p}(\Omega
),\hat{\mathbb{E}})$ forms a sublinear expectation space. Moreover, it is easy
to check that the canonical process $B$ is a $G$-Brownian motion on
$(\Omega,L_{G}^{p}(\Omega),\hat{\mathbb{E}})$ and $G(A)=\hat{\mathbb{E}%
}[\langle AB_{1},B_{1}\rangle]$ for $A\in\mathbb{S}(d)$.

Indeed, the $G$-expectation can be regarded as an upper expectation on
$L_{G}^{1}(\Omega)$.

\begin{theorem}
[\cite{D-H-P,H-P}]\label{the1.1} There exists a weakly compact set
$\mathcal{P}$ of probability measures on $(\Omega,\mathcal{B}(\Omega))$ such
that
\[
\hat{\mathbb{E}}[\xi]=\sup_{P\in\mathcal{P}}E_{P}[\xi],\ \ \ \ \text{for
all}\ \xi\in{L}_{G}^{1}{(\Omega)}.
\]

\end{theorem}

For this $\mathcal{P}$, we define the following capacity%
\[
c(A):=\sup_{P\in\mathcal{P}}P(A),\ A\in\mathcal{B}(\Omega).
\]
A set $A\subset\mathcal{B}(\Omega)$ is polar if $c(A)=0$. A property holds
\textquotedblleft$quasi$-$surely$\textquotedblright\ (q.s.) if it holds
outside a polar set. In the following, we do not distinguish two random
variables $X$ and $Y$ if $X=Y$ q.s.

We set
\[
\mathcal{L}(\Omega):=\{X\in\mathcal{B}(\Omega):E_{P}[X]\ \text{exists for each
}\ P\in\mathcal{P}\}.
\]
Then the $G$-expectation can be extended to the space $\mathcal{L}(\Omega)$
and we still denote it by $\hat{\mathbb{E}}$, i.e.,
\[
\hat{\mathbb{E}}[X]:=\sup_{P\in\mathcal{P}}E_{P}[X],\ \ \ \ \text{for
each}\ X\in\mathcal{L}(\Omega).
\]

\begin{definition}
A real function $X$ on $\Omega$ is said to be quasi-continuous if for each
$\varepsilon>0$, there exists an open set $O$ with $c(O)<\varepsilon$ such
that $X|_{O^{c}}$ is continuous.
\end{definition}

\begin{definition}
We say that $X:\Omega\mapsto\mathbb{R}$ has a quasi-continuous version if
there exists a quasi-continuous function $Y:\Omega\mapsto\mathbb{R}$ such that
$X=Y$, q.s.
\end{definition}

Then we have the following characterization of the space $L_{G}^{p}(\Omega)$,
which can be seen as a counterpart of Lusin's theorem in the nonlinear
expectation theory.

\begin{theorem}
[\cite{D-H-P}]\label{LG-ch}For each $p\geq1$, we have
\[
L_{G}^{p}(\Omega)=\{X\in\mathcal{B}(\Omega)\ :\ \ \lim\limits_{N\rightarrow
\infty}\mathbb{\hat{E}}[|X|^{p}I_{\{|X|\geq N\}}]=0\ \text{and}\ X\ \text{has
a quasi-continuous version}\}.
\]

\end{theorem}

Note that the monotone convergence theorem is different from the classical
case due to the nonlinearity.

\begin{proposition}
\label{DCT} Suppose $X_{n}$, $n\geq1$ and $X$ are $\mathcal{B}(\Omega)$-measurable.

\begin{description}
\item[(1)] Assume $X_{n}\uparrow X$ q.s. and $E_{P}[X_{1}^{-}]<\infty$ for all
$P\in\mathcal{P}$. Then $\mathbb{\hat{E}}[X_{n}]\uparrow\mathbb{\hat{E}}[X]. $

\item[(2)] If $\{X_{n}\}_{n=1}^{\infty}$ in ${L}_{G}^{1}(\Omega)$ satisfies
that $X_{n}\downarrow X$, q.s., then $\mathbb{\hat{E}}[X_{n}]\downarrow
\mathbb{\hat{E}}[X].$
\end{description}
\end{proposition}

\begin{proposition}
[Jensen's inequality]Let $X\in\mathcal{L}(\Omega)$ and $\varphi:\mathbb{R}%
\rightarrow\mathbb{R}$ be a convex function. Assume that $\mathbb{\hat{E}%
}[|X|]<\infty$ and $\varphi(X)\in\mathcal{L}(\Omega)$. Then
\[
\mathbb{\hat{E}}[\varphi(X)]\geq\varphi(\mathbb{\hat{E}}[X]).
\]

\end{proposition}

\begin{proof}
We can take $P_{k}\in\mathcal{P}$ such that $E_{P_{k}}[X]\rightarrow
\mathbb{\hat{E}}[X].$ Note that convex function is continuous, then by the
classical Jensen's inequality,
\begin{align*}
\varphi(\mathbb{\hat{E}}[X])=\lim_{k\rightarrow\infty}\varphi(E_{P_{k}%
}[X])\leq\lim_{k\rightarrow\infty}E_{P_{k}}[\varphi(X)]\leq\mathbb{\hat{E}%
}[\varphi(X)].
\end{align*}

\end{proof}

\begin{remark}
\upshape{
Since $\mathbb{\hat{E}}$ is the upper-expectation, we cannot expect that
Jensen's inequality $\mathbb{\hat{E}}[\varphi(X)]\leq\varphi(\mathbb{\hat{E}
}[X])$\ holds for concave functions. For example, we take $d=1$,
$\varphi=-x,X=|B_{1}|^{2}$ with $-\hat{\mathbb{E}}[-|B_{1}|^{2}]<\hat
{\mathbb{E}}[|B_{1}|^{2}]$. Then
\[
\mathbb{\hat{E}}[-X]>-\hat{\mathbb{E}}[X].
\]
}
\end{remark}

For each $1\leq i,j\leq d$, we denote by $\langle B^{i},B^{j}\rangle$ the
mutual quadratic variation process. Then for two processes $\eta\in M_{G}%
^{2}(0,T)$ and $\xi\in M_{G}^{1}(0,T)$, the $G$-It\^{o} integrals $\int%
_{0}^{T}\eta_{t}dB_{t}^{i}$ and $\int_{0}^{T}\xi_{t}d\langle B^{i}%
,B^{j}\rangle_{t}$ are well defined. Moreover, we have $\int_{0}^{T}\eta
_{t}dB_{t}^{i}\in L_{G}^{2}(\Omega)$ and $\int_{0}^{T}\xi_{t}d\langle
B^{i},B^{j}\rangle_{t}\in L_{G}^{1}(\Omega)$.




\begin{definition}
A process $\{M_{t}\}$ with values in $L^{1}_{G}(\Omega)$ is called a
$G$-martingale if $M_{t}\in L^{1}_{G}(\Omega_{t})$ and $\hat{\mathbb{E}}%
_{s}(M_{t})=M_{s}$ for any $s\leq t$. If $\{M_{t}\}$ and $\{-M_{t}\}$ are both
$G$-martingales, we call $\{M_{t}\}$ a symmetric $G$-martingale.
\end{definition}

We say that the function $G$ is non-degenerate if there exists a constant
$\underline{\sigma}^{2}>0$ such that
\begin{equation}
G(A)-G(A^{\prime})\geq\frac{1}{2}\underline{\sigma}^{2}\text{tr}[A-A^{\prime
}],\text{ for }A\geq A^{\prime}. \label{Myeq2.1}%
\end{equation}

\begin{remark}
\label{Myre2.1}
\upshape{  By the Hahn-Banach theorem, one can check that there exists a
	bounded, convex and closed subset $\Gamma\subset\mathbb{S}_{+}(d)$ such that	\begin{equation}
	G(A)=\frac{1}{2}\sup_{\gamma\in\Gamma}\text{\textrm{tr}}[\gamma A],\ \ \text{
		for }A\in\mathbb{S}(d), \label{neweq-1}	\end{equation}
	where $\mathbb{S}_{+}(d)$ denotes the collection of nonnegative elements in
	$\mathbb{S}(d)$. Then
	 (\ref{Myeq2.1}) is equivalent to the condition that
	$$
	\gamma\geq \underline{\sigma}	^{2}I_{d\times d},\ \text{for each}\ \gamma\in\Gamma.
	$$
In the one-dimensional case, the non-degenerate condition reduces to the condition that the lower variance of $B$ is strictly positive, i.e.,
$-\hat{\mathbb{E}}[-|B_{1}|^{2}]>0$.	
}
\end{remark}

Now we give the Girsanov theorem under the non-degenerate condition. Given
$T>0$ and $h\in M_{G}^{2}(0,T;\mathbb{R}^{d})$. We define, for $0\leq t\leq
T$,%
\begin{align*}
&  \mathcal{E}(h)_{t}:=\exp\left(  \int_{0}^{t}\langle h_{s},dB_{s}%
\rangle-\frac{1}{2}\int_{0}^{t}\langle h_{s}h_{s}^{T},d\langle{B}\rangle
_{s}\rangle\right)  ,\\
&  \tilde{B}_{t}:=B_{t}-\int_{0}^{t}d\langle{B}\rangle_{s}h_{s},
\end{align*}
where $\langle\cdot,\cdot\rangle$ is the Euclid inner product for vectors and
matrices. We set
\[
\tilde{\mathcal{H}}:=\{\varphi(\tilde{B}_{t_{1}},\tilde{B}_{t_{2}}%
\cdots,\tilde{B}_{t_{n}}):n\in\mathbb{N},0\leq t_{1}<t_{2}\cdots<t_{n}\leq
T,\varphi\in C_{b.Lip}(\mathbb{R}^{n\times d})\}.
\]
We define a sublinear expectation $\tilde{\mathbb{E}}$ by
\[
\tilde{\mathbb{E}}[\xi]:=\hat{\mathbb{E}}[\xi\mathcal{E}(h)_{T}],\ \text{for
}\xi\in\ \tilde{\mathcal{H}}.
\]

We shall assume the following $G$-Novikov's condition:

\begin{itemize}
\item[($H$)] There exists some constant $\delta>0$ such that
\begin{equation}
{\mathbb{\hat{E}}}\left[  \exp\left(  \frac{1}{2}(1+\delta)\int_{0}^{T}\langle
h_{t}h_{t}^{T},d\langle{B}\rangle_{t}\rangle\right)  \right]  <\infty
.\label{eq1-1}%
\end{equation}

\end{itemize}

Girsanov theorem for $G$-Brownian motion is stated as follows.

\begin{theorem}
\label{Gt0} \cite{XSZ,Osuka} If $G$ is non-degenerate and $h$ satisfies the
$G$-Novikov's condition (H), then the process $(\tilde{B}_{t})_{0\leq t\leq
T}$ is a $G$-Brownian motion on the sublinear expectation space $(\Omega
,\tilde{\mathcal{H}},\tilde{\mathbb{E}})$.
\end{theorem}

The $G$-Novikov's condition guarantees that $(\mathcal{E}(h)_{t})_{0\leq t\leq
T}$ is a symmetric $G$-martingale. It worth noting that the non-degenerate
assumption is not needed here.

\begin{proposition}
\cite{XSZ,Osuka}\label{Gt0-1} If the $G$-Novikov's condition holds, then
$(\mathcal{E}(h)_{t})_{0\leq t\leq T}$ is a symmetric $G$-martingale on
$(\Omega,L_{G}^{1}(\Omega),\hat{\mathbb{E}})$.
\end{proposition}

\section{Main results}

We first present a convergence theorem for sequences of random variables in
the following exponential form.

\begin{proposition}
\label{Myth3.1} Let $h\in M_{G}^{2}(0,T;\mathbb{R}^{d})$ such that
$\mathbb{\hat{E}}[\exp\left(  \delta_{0}\int_{0}^{T}\langle h_{t}h_{t}%
^{T},d\langle{B}\rangle_{t}\rangle\right)  ]<\infty$ for some $\delta_{0}>0.$
For any fixed $\alpha,\beta\in\mathbb{R}$, we denote%
\[
J_{\varepsilon}:=\exp\left(  \alpha\varepsilon\int_{0}^{T}\langle h_{t}%
,dB_{t}\rangle-\frac{\beta\varepsilon^{2}}{2}\int_{0}^{T}\langle h_{t}%
h_{t}^{T},d\langle{B}\rangle_{t}\rangle\right)  ,\text{ for }\varepsilon>0.
\]
Then
\begin{equation}
\mathbb{\hat{E}}\left[  J_{\varepsilon}\right]  \rightarrow1,\text{ as
}\varepsilon\downarrow0,\label{Myeq3.1}%
\end{equation}
and
\begin{equation}
\mathbb{\hat{E}}\left[  -J_{\varepsilon}\right]  \rightarrow-1,\text{ as
}\varepsilon\downarrow0.\label{Myeq3.6}%
\end{equation}

\end{proposition}

\begin{proof}
\textit{Part I: Proof of (\ref{Myeq3.1}).}

We first show that $\lim_{\varepsilon\rightarrow0}\mathbb{\hat{E}}\left[
J_{\varepsilon}\right]  \leq1.$ Let any $q>1$ be given and $q^{\prime}$ be the
corresponding H\"{o}lder conjugate. Denote $\alpha_{\varepsilon}=\frac
{q\alpha^{2}\varepsilon^{2}}{2}$. Then by H\"{o}lder's inequality, we have
\begin{align*}
&  \mathbb{\hat{E}}\left[  J_{\varepsilon}\right]  =\mathbb{\hat{E}}\left[
\exp\left(  \alpha\varepsilon\int_{0}^{T}\langle h_{t},dB_{t}\rangle
-\frac{\beta\varepsilon^{2}}{2}\int_{0}^{T}\langle h_{t}h_{t}^{T},d\langle
{B}\rangle_{t}\rangle\right)  \right]  \\
&  \leq\mathbb{\hat{E}}\left[  \left(  \exp\left(  \alpha\varepsilon\int%
_{0}^{T}\langle h_{t},dB_{t}\rangle-\alpha_{\varepsilon}\int_{0}^{T}\langle
h_{t}h_{t}^{T},d\langle{B}\rangle_{t}\rangle\right)  \right)  ^{q}\right]
^{\frac{1}{q}}\mathbb{\hat{E}}\left[  \left(  \exp\left(  (\alpha
_{\varepsilon}-\frac{1}{2}\beta\varepsilon^{2})\int_{0}^{T}\langle h_{t}%
h_{t}^{T},d\langle{B}\rangle_{t}\rangle\right)  \right)  ^{q^{\prime}}\right]
^{\frac{1}{q^{\prime}}}.
\end{align*}
Note that $\mathbb{\hat{E}}[\exp\left(  \delta_{0}\int_{0}^{T}\langle
h_{t}h_{t}^{T},d\langle{B}\rangle_{t}\rangle\right)  ]<\infty$ for some
$\delta_{0}>0$ implies $\mathbb{\hat{E}}[\exp\left(  \delta^{\prime}\int%
_{0}^{T}\langle h_{t}h_{t}^{T},d\langle{B}\rangle_{t}\rangle\right)  ]<\infty$
for each $\delta^{\prime}\leq\delta_{0}.$ Then applying Proposition
\ref{Gt0-1}, we get%
\begin{align*}
&  \mathbb{\hat{E}}\left[  \left(  \exp\left(  \alpha\varepsilon\int_{0}%
^{T}\langle h_{t},dB_{t}\rangle-\alpha_{\varepsilon}\int_{0}^{T}\langle
h_{t}h_{t}^{T},d\langle{B}\rangle_{t}\rangle\right)  \right)  ^{q}\right]  \\
&  =\mathbb{\hat{E}}\left[  \exp\left(  \alpha\varepsilon q\int_{0}^{T}\langle
h_{t},dB_{t}\rangle-\frac{\alpha^{2}\varepsilon^{2}q^{2}}{2}\int_{0}%
^{T}\langle h_{t}h_{t}^{T},d\langle{B}\rangle_{t}\rangle\right)  \right]
=1,\text{ when }\varepsilon>0\text{ is small}.
\end{align*}
Thus,%
\[
\mathbb{\hat{E}}\left[  J_{\varepsilon}\right]  \leq\mathbb{\hat{E}}\left[
\left(  \exp\left(  (\alpha_{\varepsilon}-\frac{1}{2}\beta\varepsilon^{2}%
)\int_{0}^{T}\langle h_{t}h_{t}^{T},d\langle{B}\rangle_{t}\rangle\right)
\right)  ^{q^{\prime}}\right]  ^{\frac{1}{q^{\prime}}},\text{ when
}\varepsilon>0\text{ is small}.
\]
It remains to show that
\begin{align*}
\mathbb{\hat{E}}\left[  \left(  \exp\left(  (\alpha_{\varepsilon}-\frac{1}%
{2}\beta\varepsilon^{2})\int_{0}^{T}\langle h_{t}h_{t}^{T},d\langle{B}%
\rangle_{t}\rangle\right)  \right)  ^{q^{\prime}}\right]   &  =\mathbb{\hat
{E}}\left[  \left(  \exp\left(  \frac{1}{2}\varepsilon^{2}(q\alpha^{2}%
-\beta)\int_{0}^{T}\langle h_{t}h_{t}^{T},d\langle{B}\rangle_{t}%
\rangle\right)  \right)  ^{q^{\prime}}\right]  \\
&  \rightarrow1,\text{ as }\varepsilon\downarrow0\text{.}%
\end{align*}
If $q\alpha^{2}-\beta\leq0,$ since
\[
\left(  \exp\left(  \frac{1}{2}\varepsilon^{2}(q\alpha^{2}-\beta)\int_{0}%
^{T}\langle h_{t}h_{t}^{T},d\langle{B}\rangle_{t}\rangle\right)  \right)
^{q^{\prime}}\uparrow1,\text{ as }\varepsilon\downarrow0\text{,}%
\]
we have, by Proposition \ref{DCT} (1),
\[
\mathbb{\hat{E}}\left[  \left(  \exp\left(  \frac{1}{2}\varepsilon^{2}%
(q\alpha^{2}-\beta)\int_{0}^{T}\langle h_{t}h_{t}^{T},d\langle{B}\rangle
_{t}\rangle\right)  \right)  ^{q^{\prime}}\right]  \uparrow1,\text{ as
}\varepsilon\downarrow0\text{.}%
\]
If $q\alpha^{2}-\beta\geq0$, from Theorem \ref{LG-ch} and the assumption that
$\mathbb{\hat{E}}[\exp\left(  \delta_{0}\int_{0}^{T}\langle h_{t}h_{t}%
^{T},d\langle{B}\rangle_{t}\rangle\right)  ]<\infty$ for some $\delta_{0}>0,$
it is easy to see that
\[
\left(  \exp\left(  \frac{1}{2}\varepsilon^{2}(q\alpha^{2}-\beta)\int_{0}%
^{T}\langle h_{t}h_{t}^{T},d\langle{B}\rangle_{t}\rangle\right)  \right)
^{q^{\prime}}\in L_{G}^{1}(\Omega),\text{ for }\varepsilon>0\text{ small}.
\]
Note that
\[
\left(  \exp\left(  \frac{1}{2}\varepsilon^{2}(q\alpha^{2}-\beta)\int_{0}%
^{T}\langle h_{t}h_{t}^{T},d\langle{B}\rangle_{t}\rangle\right)  \right)
^{q^{\prime}}\downarrow1,\text{ as }\varepsilon\downarrow0\text{.}%
\]
Applying Proposition \ref{DCT} (2), we then get
\begin{equation}
\mathbb{\hat{E}}\left[  \left(  \exp\left(  \frac{1}{2}\varepsilon^{2}%
(q\alpha^{2}-\beta)\int_{0}^{T}\langle h_{t}h_{t}^{T},d\langle{B}\rangle
_{t}\rangle\right)  \right)  ^{q^{\prime}}\right]  \downarrow1,\text{ as
}\varepsilon\downarrow0\text{.}%
\end{equation}

Now we prove that $\lim_{\varepsilon\rightarrow0}\mathbb{\hat{E}}\left[
J_{\varepsilon}\right]  \geq1.$ From Jensen's inequality, we get%
\begin{align*}
\mathbb{\hat{E}}\left[  J_{\varepsilon}\right]   &  =\mathbb{\hat{E}}\left[
\exp\left(  \alpha\varepsilon\int_{0}^{T}\langle h_{t},dB_{t}\rangle
-\frac{\beta\varepsilon^{2}}{2}\int_{0}^{T}\langle h_{t}h_{t}^{T},d\langle
{B}\rangle_{t}\rangle\right)  \right]  \\
&  \geq\exp\left(  \mathbb{\hat{E}}\left[  \alpha\varepsilon\int_{0}%
^{T}\langle h_{t},dB_{t}\rangle-\frac{\beta\varepsilon^{2}}{2}\int_{0}%
^{T}\langle h_{t}h_{t}^{T},d\langle{B}\rangle_{t}\rangle\right]  \right)  \\
&  =\exp\left(  \mathbb{\hat{E}}\left[  -\frac{\beta\varepsilon^{2}}{2}%
\int_{0}^{T}\langle h_{t}h_{t}^{T},d\langle{B}\rangle_{t}\rangle\right]
\right)  \\
&  =\exp\left(  \frac{\varepsilon^{2}}{2}\mathbb{\hat{E}}{\mathbb{[-}}%
\beta\int_{0}^{T}\langle h_{t}h_{t}^{T},d\langle{B}\rangle_{t}\rangle]\right)
\\
&  \rightarrow1,\text{ as }\varepsilon\downarrow0.
\end{align*}

\textit{Part II: Proof of (\ref{Myeq3.6}).}

Applying the classical Jensen's inequality under each $P\in\mathcal{P}$, we
have
\begin{align*}
\mathbb{\hat{E}}\left[  -J_{\varepsilon}\right]   &  =\mathbb{\hat{E}}\left[
-\exp\left(  \alpha\varepsilon\int_{0}^{T}\langle h_{t},dB_{t}\rangle
-\frac{\beta\varepsilon^{2}}{2}\int_{0}^{T}\langle h_{t}h_{t}^{T},d\langle
{B}\rangle_{t}\rangle\right)  \right]  \\
&  =\sup_{P\in\mathcal{P}}E_{P}\left[  -\exp\left(  \alpha\varepsilon\int%
_{0}^{T}\langle h_{t},dB_{t}\rangle-\frac{\beta\varepsilon^{2}}{2}\int_{0}%
^{T}\langle h_{t}h_{t}^{T},d\langle{B}\rangle_{t}\rangle\right)  \right]  \\
&  \leq\sup_{P\in\mathcal{P}}\left\{  -\exp\left(  E_{P}\left[  \alpha
\varepsilon\int_{0}^{T}\langle h_{t},dB_{t}\rangle-\frac{\beta\varepsilon^{2}%
}{2}\int_{0}^{T}\langle h_{t}h_{t}^{T},d\langle{B}\rangle_{t}\rangle\right]
\right)  \right\}  \\
&  =\sup_{P\in\mathcal{P}}\left\{  -\exp\left(  E_{P}\left[  -\frac
{\beta\varepsilon^{2}}{2}\int_{0}^{T}\langle h_{t}h_{t}^{T},d\langle{B}%
\rangle_{t}\rangle\right]  \right)  \right\}  \\
&  =\sup_{P\in\mathcal{P}}\left\{  -\exp\left(  -\frac{\beta\varepsilon^{2}%
}{2}E_{P}\left[  \int_{0}^{T}\langle h_{t}h_{t}^{T},d\langle{B}\rangle
_{t}\rangle\right]  \right)  \right\}  .
\end{align*}
Since $y\rightarrow-\exp(-y)$ is increasing, we further get
\begin{align*}
\mathbb{\hat{E}}\left[  -J_{\varepsilon}\right]   &  \leq\sup_{P\in
\mathcal{P}}\left\{  -\exp\left(  -\frac{\beta\varepsilon^{2}}{2}E_{P}\left[
\int_{0}^{T}\langle h_{t}h_{t}^{T},d\langle{B}\rangle_{t}\rangle\right]
\right)  \right\}  \\
&  =-\exp\left(  -\sup_{P\in\mathcal{P}}\frac{\beta\varepsilon^{2}}{2}%
E_{P}\left[  \int_{0}^{T}\langle h_{t}h_{t}^{T},d\langle{B}\rangle_{t}%
\rangle\right]  \right)  \\
&  =-\exp\left(  -\frac{\varepsilon^{2}}{2}\mathbb{\hat{E}}\left[  \beta
\int_{0}^{T}\langle h_{t}h_{t}^{T},d\langle{B}\rangle_{t}\rangle\right]
\right)  \\
&  \rightarrow-1,\text{ as }\varepsilon\downarrow0.
\end{align*}
Moreover, from Part I, we know that
\[
\mathbb{\hat{E}}\left[  -J_{\varepsilon}\right]  \geq-\mathbb{\hat{E}%
}[J_{\varepsilon}]\rightarrow-1,\text{ as }\varepsilon\downarrow0.
\]

\end{proof}

The main result of our paper is the following Girsanov theorem for
$G$-Brownian motion in the degenerate case.

Let $h\in M_{G}^{2}(0,T;\mathbb{R}^{d}).$ We define
\[
\tilde{B}_{t}:=B_{t}-\int_{0}^{t}d\langle{B}\rangle_{s}h_{s},\ \text{for}%
\ 0\leq t\leq T,
\]
and
\[
\tilde{\mathcal{H}}:=\{\varphi(\tilde{B}_{t_{1}},\tilde{B}_{t_{2}}%
\cdots,\tilde{B}_{t_{n}}):n\in\mathbb{N},0\leq t_{1}<t_{2}\cdots<t_{n}\leq
T,\varphi\in C_{b.Lip}(\mathbb{R}^{n\times d})\}.
\]

\begin{theorem}
\label{Gt1} Assume that for $h\in M_{G}^{2}(0,T;\mathbb{R}^{d}),$ the
$G$-Novikov's condition (H) holds for some $\delta>0$ and $\mathbb{\hat{E}%
}\left[  \exp\left(  \int_{0}^{T}\delta_{0}|h_{t}|^{2}dt)\right)  \right]
<\infty$ for some $\delta_{0}>0$. Define a sublinear expectation
$\tilde{\mathbb{E}}$ by
\[
\tilde{\mathbb{E}}[\xi]:=\hat{\mathbb{E}}[\xi\mathcal{E}(h)_{T}],\ \text{for
}\xi\in\ \tilde{\mathcal{H}}.\
\]
Then the process $(\tilde{B}_{t})_{t\geq0}$ is a $G$-Brownian motion on the
sublinear expectation space $(\Omega,\tilde{\mathcal{H}},\tilde{\mathbb{E}})$.
\end{theorem}

\begin{remark}\label{Rm1.1}
	\upshape{
\begin{description}
	\item[{\rm (i)}]
	Compared with Theorem \ref{Gt0},  we have imposed in Theorem \ref{Gt1} an additional
	assumption that $\mathbb{\hat{E}}\left[\exp\left(  \int_{0}^{T}\delta_0 |h_{s}|	^{2}ds)\right) \right ]<\infty$ for some $\delta_0>0$. In the non-degenerate case, this assumption is
	implied by the $G$-Novikov's condition by noting that, from Corollary 5.7 in Chapter III of \cite{Peng 1} and Remark \ref{Myre2.1}, $\frac{d\langle{B}	\rangle_{t}}{dt}\geq\underline{\sigma}^{2}I_{d\times d}$. But in the degenerate case, it is needed for our arguments.
	
	\item[{\rm (ii)}]
	According to  the proofs of Lemma 2.2 in
	\cite{XSZ} and  Proposition 5.10 in \cite{Osuka}, the $G$-Novikov's
	implies that $\hat{\mathbb{E}}[|\mathcal{E}(B)_{T}|^{p}]<\infty$ for some
	$p>1$. This property will be used in the proof of the main theorem.
\end{description}
		}
\end{remark}

To prove Theorem \ref{Gt1}, it suffices to show for $t_{1}\leq t_{2}\leq
\cdots\leq t_{n}\leq T$ and $\varphi\in C_{b.Lip}(\mathbb{R}^{n\times d})$, it
holds that
\begin{equation}
\hat{\mathbb{E}}[\varphi(B_{t_{1}},B_{t_{2}},\cdots,B_{t_{n}})]=\tilde
{\mathbb{E}}[\varphi(\tilde{B}_{t_{1}},\tilde{B}_{t_{2}},\cdots,\tilde
{B}_{t_{n}})]. \label{Myeq3.4}%
\end{equation}

Since $B$ is possibly degenerate, we use the following product space method in
the nonlinear expectation setting to add a small linear Brownian motion term
to $B$, so to get a non-degenerate perturbation $B^{\varepsilon}$.

Let
\[
\bar{G}(A^{\prime})=G(A)+\frac{1}{2}\text{tr}[C],\ \text{ for}\ A^{\prime
}=\left[
\begin{array}
[c]{cc}%
A & B\\
B & C
\end{array}
\right]  \in\mathbb{S}(2d),\text{ where }A,B,C\in\mathbb{S}(d).
\]
Following the method in Section 2, we can construct an auxiliary $\bar{G}%
$-expectation space $(\bar{\Omega},L_{\bar{G}}^{1}(\bar{\Omega}),\mathbb{\bar
{E}})$ such that

\begin{description}
\item[(i)] $\bar{\Omega}=\Omega\times C_{0}([0,\infty);\mathbb{R}^{d})$;

\item[(ii)] $\bar{B}_{t}:=(B_{t},W_{t})_{t\geq0}$ is a $2d$-dimensional
$\bar{G}$-Brownian motion, where $W$ is the canonical process on
$C_{0}([0,\infty);\mathbb{R}^{d}).$
\end{description}

Moreover, by the definition of $\mathbb{\bar{E}}$, we also have:

\begin{lemma}
Let $(\bar{\Omega},L_{\bar{G}}^{1}(\bar{\Omega}),\mathbb{\bar{E}})$ be defined
as above. Then

\begin{description}
\item[(iii)] $\mathbb{\bar{E}}=\mathbb{\hat{E}}$ on $L_{G}^{1}(\Omega)$ and
$(B_{t})_{t\geq0}$ is a $d$-dimensional ${G}$-Brownian motion under
$\mathbb{\bar{E}}$;

\item[(iv)] $(W_{t})_{t\geq0}$ is a $d$-dimensional standard Brownian motion
under $\mathbb{\bar{E}}$.
\end{description}
\end{lemma}

\begin{proof}
We only prove that $\mathbb{\bar{E}}=\mathbb{\hat{E}}$ on $L_{G}^{1}(\Omega)$,
which implies $(B_{t})_{t\geq0}$ is a ${G}$-Brownian motion under
$\mathbb{\bar{E}}$, and the proof for (iv) is similar. By Step 2 in the
definition of $G$-expectation in Section 2, we only need to show that, for any
given $X=\varphi(B_{t+s}-B_{s})$, where $\varphi\in C_{b.Lip}(\mathbb{R}^{d}%
)$, we have
\begin{equation}
\mathbb{\bar{E}}[X]=\mathbb{\hat{E}}{[X]}. \label{Myeq3.12}%
\end{equation}
From Step 1 in the definition of $G$-expectation, we know that
\[
\mathbb{\bar{E}}[X]=\bar{u}(t,0,0).
\]
Here $\bar{u}(r,x_{1},x_{2})\in C([0,T]\times\mathbb{R}^{d}\times
\mathbb{R}^{d})$ is the solution of the following $\bar{G}$-heat equation:%
\begin{equation}
\partial_{t}\bar{u}-\bar{G}(D_{xx}^{2}\bar{u})=0,\ \bar{u}(0,x_{1}%
,x_{2})=\varphi(x_{1}),\ \ \ \text{where}\ x=(x_{1},x_{2}). \label{Myeq3.10}%
\end{equation}
Similarly,
\[
\mathbb{\hat{E}}[X]={u}(t,0),
\]
where ${u}(r,x_{1})\in C([0,T]\times\mathbb{R}^{d})$ is the solution of the
following ${G}$-heat equation:%
\[
\partial_{t}{u}-{G}(D_{x_{1}x_{1}}^{2}{u})=0,\ u(0,x_{1})=\varphi(x_{1}).
\]
It is easy to check that ${u}(r,x_{1})$ is also a solution of (\ref{Myeq3.10}%
). Then, from the uniqueness theorem of viscosity solutions, we get
\[
{u}(r,x_{1})=\bar{u}(r,x_{1},x_{2}),\ \text{for}\ (r,x_{1},x_{2})\in
\lbrack0,T]\times\mathbb{R}^{d}\times\mathbb{R}^{d},
\]
which implies the desired (\ref{Myeq3.12}).
\end{proof}

For each fixed $\varepsilon\in(0,1)$, we define $B_{t}^{\varepsilon}%
=B_{t}+\varepsilon W_{t}$. Following Proposition 1.4 in Chapter III of
\cite{Peng 1}, we deduce that $(B_{t}^{\varepsilon})_{t\geq0}$ is a
$d$-dimensional $G_{\varepsilon}$-Brownian motion under $\mathbb{\bar{E}}$,
where
\[
G_{\varepsilon}(A)=\mathbb{\bar{E}}[\langle AB_{1}^{\varepsilon}%
,B_{1}^{\varepsilon}\rangle]=\mathbb{\bar{E}}\left[  \left\langle \left[
\begin{array}
[c]{cc}%
A & \varepsilon A\\
\varepsilon A & \varepsilon^{2}A
\end{array}
\right]  \bar{B}_{1},\bar{B}_{1}\right\rangle \right]  =\bar{G}\left(  \left[
\begin{array}
[c]{cc}%
A & \varepsilon A\\
\varepsilon A & \varepsilon^{2}A
\end{array}
\right]  \right)  =G(A)+\frac{\varepsilon^{2}}{2}\text{tr}[A],\ \text{for}%
\ A\in\mathbb{S}(d).
\]
We claim that the $G_{\varepsilon}$ is non-degenerate. Indeed, for $A\geq B,$
we have
\[
G_{\varepsilon}(A)-G_{\varepsilon}(B)=G(A)-G(B)+\frac{\varepsilon^{2}}%
{2}\text{tr}[A-B]\geq\frac{\varepsilon^{2}}{2}\text{tr}[A-B].
\]

The following two lemmas concern respectively the quadratic variation and the
stochastic exponential of ${B}^{\varepsilon}$.

\begin{lemma}
\label{Myle3.1}We have
\begin{equation}
\langle{B}^{\varepsilon}\rangle_{t}=\langle{B}\rangle_{t}+\varepsilon
^{2}tI_{d\times d}, \label{Myeq3.2}%
\end{equation}
where $I_{d\times d}$ is the $d\times d$ identity matrix.
\end{lemma}

\begin{proof}
We can find a set $\Gamma\subset\mathbb{S}_{+}(d)$ such that%
\begin{equation}
G(A)=\frac{1}{2}\sup_{\gamma\in\Gamma}\text{\textrm{tr}}[\gamma A],\ \ \text{
for }A\in\mathbb{S}(d).
\end{equation}
Then it is easy to check that
\[
\bar{G}(A^{\prime})=\frac{1}{2}\sup_{\gamma\in\Gamma}\text{\textrm{tr}}\left[
A^{\prime}\left[
\begin{array}
[c]{cc}%
\gamma & 0\\
0 & I_{d\times d}%
\end{array}
\right]  \right]  ,\ \ \ \ \text{ for}\ A^{\prime}\in\mathbb{S}(2d).
\]
By Corollary 5.7 in Chapter III of \cite{Peng 1}, we have
\[
\left\langle \bar{B}\right\rangle _{t}=\left[
\begin{array}
[c]{cc}%
\langle{B}\rangle_{t} & \langle{B,W}\rangle_{t}\\
\langle{B,W}\rangle_{t} & \langle{W}\rangle_{t}%
\end{array}
\right]  \in\left\{  t\left[
\begin{array}
[c]{cc}%
\gamma & 0\\
0 & I_{d\times d}%
\end{array}
\right]  :\gamma\in\Gamma\right\}  .
\]
From this we deduce that $\langle{B,W}\rangle_{t}=0$, and thus,
\[
\langle{B}^{\varepsilon}\rangle_{t}=\langle{B}\rangle_{t}+2\varepsilon
\langle{B},W\rangle_{t}+\varepsilon^{2}\langle W\rangle_{t}=\langle{B}%
\rangle_{t}+\varepsilon^{2}tI_{d\times d}.
\]
This completes the proof.
\end{proof}

\begin{lemma}
\label{Myle3.6} Under the assumptions of Theorem \ref{Gt1}, the $G$-Novikov's
condition holds for ${B}^{\varepsilon}$: For any given $0<\delta^{\prime
}<\delta$, there exists some $\varepsilon_{\delta^{\prime}}>0$ such that for
each $0<\varepsilon\leq\varepsilon_{\delta^{\prime}},$
\begin{equation}
\mathbb{\bar{E}}\left[  \exp\left(  \frac{1}{2}(1+\delta^{\prime})\int_{0}%
^{T}\langle h_{t}h_{t}^{T},d\langle{B}^{\varepsilon}\rangle_{t}\rangle\right)
\right]  <\infty.
\end{equation}

\end{lemma}

\begin{proof}
We first take $p>1$ so small such that%
\[
p(1+\delta^{\prime})\leq1+\delta.
\]
Let $p^{\prime}$ be the H\"{o}lder conjugate of $p$. Then after taking
$\varepsilon>0$ small, we have \[\frac
{\varepsilon^{2}p^{\prime}}{2}(1+\delta^{\prime})\leq\delta_{0}.\] Applying the H\"{o}lder's inequality, we obtain
\begin{align*}
&  \mathbb{\bar{E}}\left[  \exp\left(  \frac{1}{2}(1+\delta^{\prime})\int%
_{0}^{T}\langle h_{t}h_{t}^{T},d\langle{B}^{\varepsilon}\rangle_{t}%
\rangle\right)  \right]  \\
&  ={\mathbb{\hat{E}}}\left[  \exp\left(  \frac{1}{2}(1+\delta^{\prime}%
)\int_{0}^{T}\langle h_{t}h_{t}^{T},d\langle{B}\rangle_{t}\rangle\right)
\exp\left(  \frac{\varepsilon^{2}}{2}(1+\delta^{\prime})\int_{0}^{T}%
|h_{t}|^{2}dt\right)  \right]  \\
&  \leq{\mathbb{\hat{E}}}\left[  \exp\left(  \frac{p}{2}(1+\delta^{\prime
})\int_{0}^{T}\langle h_{t}h_{t}^{T},d\langle{B}\rangle_{t}\rangle\right)
\right]  ^{\frac{1}{p}}{\mathbb{\hat{E}}}\left[  \exp\left(  \frac
{\varepsilon^{2}p^{\prime}}{2}(1+\delta^{\prime})\int_{0}^{T}|h_{t}%
|^{2}dt\right)  \right]  ^{\frac{1}{p^{\prime}}}\\
&  <\infty,
\end{align*}
where in the last inequality we have applied the $G$-Novikov's condition for
$B$ and the assumption that $\mathbb{\hat{E}}\left[  \exp\left(  \int_{0}%
^{T}\delta_{0}|h_{t}|^{2}dt)\right)  \right]  <\infty.$
\end{proof}

Now we are ready to state the proof of Theorem \ref{Gt1}.

\begin{proof}
We define, for $0\leq t\leq T$,
\begin{align}
&  N_{t}^{\varepsilon}:=\exp\left(  \int_{0}^{t}\langle h_{s},dB_{s}%
^{\varepsilon}\rangle-\frac{1}{2}\int_{0}^{t}\langle h_{s}h_{s}^{T}%
,d\langle{B}^{\varepsilon}\rangle_{s}\rangle\right)  =\mathcal{E}(h)_{t}%
\exp\left(  \int_{0}^{t}\varepsilon\langle h_{s},dW_{s}\rangle-\frac{1}%
{2}\varepsilon^{2}\int_{0}^{t}|h_{s}|^{2}ds\right)  ,\label{Myeq3.3}\\
&  \tilde{B}_{t}^{\varepsilon}:=B_{t}^{\varepsilon}-\int_{0}^{t}d\langle
{B}^{\varepsilon}\rangle_{s}h_{s}=B_{t}^{\varepsilon}-\int_{0}^{t}d\langle
{B}\rangle_{s}h_{s}-\varepsilon^{2}\int_{0}^{t}h_{s}ds, \label{Myeq3.9}%
\end{align}
where we have used Lemma \ref{Myle3.1} in the second equalities in
(\ref{Myeq3.3}) and (\ref{Myeq3.9}). We also define
\[
\tilde{\mathbb{E}}^{\varepsilon}[\xi]:={\mathbb{\bar{E}}}[\xi N_{T}%
^{\varepsilon}],\ \text{for }\xi\in\ \tilde{\mathcal{H}}.
\]
Since  $(B_{t}^{\varepsilon})_{t\geq0}$ is
non-degenerate and from Lemma \ref{Myle3.6}, it satisfies the $G$-Novikov's
condition for small enough $\varepsilon>0$, then we can apply Theorem \ref{Gt0} to obtain that, for $\varphi\in
C_{b.Lip}(\mathbb{R}^{n\times d})$,
\begin{equation}
\mathbb{\bar{E}}[\varphi(B_{t_{1}}^{\varepsilon},B_{t_{2}}^{\varepsilon
},\cdots,B_{t_{n}}^{\varepsilon})]=\tilde{\mathbb{E}}^{\varepsilon}%
[\varphi(\tilde{B}_{t_{1}}^{\varepsilon},\tilde{B}_{t_{2}}^{\varepsilon
},\cdots,\tilde{B}_{t_{n}}^{\varepsilon})],\text{ for }\varepsilon>0\text{ small}. \label{Myeq3.7}%
\end{equation}
To completes the proof, we shall show that the left-hand side (right-hand side
resp.) of (\ref{Myeq3.7}) converges to the left-hand side (right-hand side
resp.) of (\ref{Myeq3.4}) by the following two steps.

\textit{Step 1. The left-hand side.} By the Lipschitz continuity assumption of
$\varphi,$ we have%
\begin{align*}
&  |\mathbb{\bar{E}}[\varphi(B_{t_{1}}^{\varepsilon},B_{t_{2}}^{\varepsilon
},\cdots,B_{t_{n}}^{\varepsilon})]-\hat{\mathbb{E}}[\varphi(B_{t_{1}}%
,B_{t_{2}},\cdots,B_{t_{n}})]|\\
&  =|\mathbb{\bar{E}}[\varphi(B_{t_{1}}^{\varepsilon},B_{t_{2}}^{\varepsilon
},\cdots,B_{t_{n}}^{\varepsilon})]-\mathbb{\bar{E}}[\varphi(B_{t_{1}}%
,B_{t_{2}},\cdots,B_{t_{n}})]|\\
&  \leq L_{\varphi}\mathbb{\bar{E}[}|B_{t_{1}}^{\varepsilon}-B_{t_{1}%
}|+|B_{t_{2}}^{\varepsilon}-B_{t_{2}}|+\cdots+|B_{t_{n}}^{\varepsilon
}-B_{t_{n}})|]\\
&  =L_{\varphi}\varepsilon\mathbb{\bar{E}[}|W_{t_{1}}|+|W_{t_{2}}%
|+\cdots+|W_{t_{n}}|]\rightarrow0,\text{ as }\varepsilon\rightarrow0,
\end{align*}
where $L_{\varphi}$ is the Lipschitz constant of $\varphi$.

\textit{Step 2. The right-hand side. }
Let $p>1$ be the constant in Remark \ref{Rm1.1} (ii). Then
 from the definition of $\tilde
{\mathbb{E}}^{\varepsilon},$ we have
\begin{align*}
&  |\tilde{\mathbb{E}}^{\varepsilon}[\varphi(\tilde{B}_{t_{1}}^{\varepsilon
},\tilde{B}_{t_{2}}^{\varepsilon},\cdots,\tilde{B}_{t_{n}}^{\varepsilon
})]-\tilde{\mathbb{E}}[\varphi(\tilde{B}_{t_{1}},\tilde{B}_{t_{2}}%
,\cdots,\tilde{B}_{t_{n}})]|\\
&  =|{\mathbb{\bar{E}}}[\varphi(\tilde{B}_{t_{1}}^{\varepsilon},\tilde
{B}_{t_{2}}^{\varepsilon},\cdots,\tilde{B}_{t_{n}}^{\varepsilon}%
)N_{T}^{\varepsilon}]-\hat{\mathbb{E}}[\varphi(\tilde{B}_{t_{1}},\tilde
{B}_{t_{2}},\cdots,\tilde{B}_{t_{n}})\mathcal{E}(h)_{T}]|\\
&  \leq{\mathbb{\bar{E}}}[|\varphi(\tilde{B}_{t_{1}}^{\varepsilon},\tilde
{B}_{t_{2}}^{\varepsilon},\cdots,\tilde{B}_{t_{n}}^{\varepsilon}%
)N_{T}^{\varepsilon}-\varphi(\tilde{B}_{t_{1}}^{\varepsilon},\tilde{B}_{t_{2}%
}^{\varepsilon},\cdots,\tilde{B}_{t_{n}}^{\varepsilon})\mathcal{E}(h)_{T}|]\\
&  +{\mathbb{\bar{E}}}[|\varphi(\tilde{B}_{t_{1}}^{\varepsilon},\tilde
{B}_{t_{2}}^{\varepsilon},\cdots,\tilde{B}_{t_{n}}^{\varepsilon}%
)\mathcal{E}(h)_{T}-\varphi(\tilde{B}_{t_{1}},\tilde{B}_{t_{2}},\cdots
,\tilde{B}_{t_{n}})\mathcal{E}(h)_{T}|]\\
&  \leq C_{\varphi}{\mathbb{\bar{E}}}[|N_{T}^{\varepsilon}-\mathcal{E}%
(h)_{T}|]\\
&  +{\mathbb{\bar{E}}}[|\varphi(\tilde{B}_{t_{1}}^{\varepsilon},\tilde
{B}_{t_{2}}^{\varepsilon},\cdots,\tilde{B}_{t_{n}}^{\varepsilon}%
)-\varphi(\tilde{B}_{t_{1}},\tilde{B}_{t_{2}},\cdots,\tilde{B}_{t_{n}%
})|^{p^{\prime}}]^{\frac{1}{p^{\prime}}}\hat{\mathbb{E}}{\mathbb{[}%
}|\mathcal{E}(h)_{T}|^{p}]^{\frac{1}{p}}\\
&  =:I_{1}+I_{2},
\end{align*}
where $C_{\varphi}$ is the bound of $\varphi$ and $p^{\prime}$ is the
H\"{o}lder conjugate of $p.$

Now we show that $I_{1},I_{2}\rightarrow0,$ as $\varepsilon\rightarrow0.$ The
proof of $I_{2}\rightarrow0$\ is similar to that of the left-hand side in Step
1, so we omit it, and we only need to consider the $I_{1}$ term. By
H\"{o}lder's inequality, we get
\begin{equation}%
\begin{split}
{\mathbb{\bar{E}}}[|N_{T}^{\varepsilon}-\mathcal{E}(h)_{T}|]  &
={\mathbb{\bar{E}}}\left[  \mathcal{E}(h)_{T}\left\vert \exp\left(  \int%
_{0}^{T}\varepsilon\langle h_{s},dW_{s}\rangle-\frac{1}{2}\varepsilon^{2}%
\int_{0}^{T}|h_{s}|^{2}ds\right)  -1\right\vert \right] \\
&  \leq{\mathbb{\bar{E}}}[|\mathcal{E}(h)_{T}|^{p}]^{\frac{1}{p}}%
{\mathbb{\bar{E}}}\left[  \left\vert \exp\left(  \int_{0}^{t}\varepsilon
\langle h_{s},dW_{s}\rangle-\frac{1}{2}\varepsilon^{2}\int_{0}^{t}|h_{s}%
|^{2}ds\right)  -1\right\vert ^{p^{\prime}}\right]  ^{\frac{1}{p^{\prime}}}%
\end{split}
\label{Myeq3.8}%
\end{equation}

Let any $r\geq0$ be fixed. From the assumption that ${\mathbb{\bar{E}}}\left[
\exp\left(  \int_{0}^{T}\delta_{0}|h_{s}|^{2}ds)\right)  \right]
=\mathbb{\hat{E}}\left[  \exp\left(  \int_{0}^{T}\delta_{0}|h_{s}%
|^{2}ds)\right)  \right]  <\infty$ for some $\delta_{0}>0$ and Proposition
\ref{Myth3.1}, we have
\begin{align*}
&  {\mathbb{\bar{E}}}\left[  \pm\left(  \exp\left(  \int_{0}^{t}%
\varepsilon\langle h_{s},dW_{s}\rangle-\frac{1}{2}\varepsilon^{2}\int_{0}%
^{t}|h_{s}|^{2}ds\right)  \right)  ^{r}\right] \\
&  ={\mathbb{\bar{E}}}\left[  \pm\exp\left(  r\varepsilon\int_{0}^{t}\langle
h_{s},dW_{s}\rangle-\frac{1}{2}r\varepsilon^{2}\int_{0}^{t}|h_{s}%
|^{2}ds\right)  \right]  \rightarrow\pm1,\text{ as }\varepsilon\downarrow0.
\end{align*}
Then applying the binomial theorem, we get
\begin{align*}
&  {\mathbb{\bar{E}}}\left[  \left\vert \exp\left(  \int_{0}^{t}%
\varepsilon\langle h_{s},dW_{s}\rangle-\frac{1}{2}\varepsilon^{2}\int_{0}%
^{t}|h_{s}|^{2}ds\right)  -1\right\vert ^{p^{\prime}}\right]  ^{\frac
{1}{p^{\prime}}}\\
&  \leq{\mathbb{\bar{E}}}\left[  \left(  \exp\left(  \int_{0}^{t}%
\varepsilon\langle h_{s},dW_{s}\rangle-\frac{1}{2}\varepsilon^{2}\int_{0}%
^{t}|h_{s}|^{2}ds\right)  -1\right)  ^{N}\right]  ^{\frac{1}{N}}\\
&  \leq\left\{  \sum_{k=0}^{N}C_{k}^{N}{\mathbb{\bar{E}}}\left[  \left(
\exp\left(  \int_{0}^{t}\varepsilon\langle h_{s},dW_{s}\rangle-\frac{1}%
{2}\varepsilon^{2}\int_{0}^{t}|h_{s}|^{2}ds\right)  \right)  ^{N-k}%
(-1)^{k}\right]  \right\}  ^{\frac{1}{N}}\\
&  \rightarrow\left\{  \sum_{k=0}^{N}C_{k}^{N}(-1)^{k}\right\}  ^{\frac{1}{N}%
}\\
&  =\left\{  (1-1)^{N}\right\}  ^{\frac{1}{N}}\\
&  =0,\text{ as }\varepsilon\downarrow0,
\end{align*}
where $N$\ is an even number not smaller than $p^{\prime}.$ Therefore,
combining this with (\ref{Myeq3.8}), we obtain
\[
{\mathbb{\bar{E}}}[|N_{T}^{\varepsilon}-\mathcal{E}(h)_{T}|]\rightarrow
0,\text{ as }\varepsilon\downarrow0,
\]
which implies
\[
I_{1}\rightarrow0,\text{ as }\varepsilon\downarrow0,
\]
as desired.
\end{proof}

\bigskip

\noindent\textbf{Acknowledgement}: The author would like to thank the
anonymous referee for the careful reading and valuable comments which improved
the presentation of this manuscript.


\begin{thebibliography}{9}                                                                                                %


\bibitem {D-H-P}L. Denis, M. Hu and S. Peng, Function spaces and capacity
related to a sublinear expectation: application to $G$-Brownian motion paths,
Potential Anal. 34 (2011) 139--161.

\bibitem {Gao}F. Gao, Pathwise properties and homeomorphic flows for
stochastic differential equations driven by $G$-Brownian motion. Stochastic
Process. Appl. 119 (2009), no. 10, 3356-3382.

\bibitem {HJPS}M. Hu, S. Ji, S. Peng and Y. Song, {Backward stochastic
differential equations driven by $G$-Brownian motion}. Stochastic Process.
Appl. 124, 759-784, 2014.

\bibitem {H-P}M. Hu and S. Peng, On representation theorem of $G$-expectations
and paths of $G$-Brownian motion, Acta Math. Appl. Sin. Engl. Ser. 25 (2009) 539--546.

\bibitem {Osuka}E. Osuka, Girsanov's formula for $G$-Brownian motion.
Stochastic Process. Appl. 123 (2013), no. 4, 1301-1318.

\bibitem {peng2005}S. Peng, $G$-expectation, $G$-Brownian motion and related
stochastic calculus of It\^{o} type, in: Stochastic Analysis and Applications,
in: Abel Symp., vol. 2, 2007, pp. 541--567.

\bibitem {peng2008}S. Peng, Multi-dimensional $G$-Brownian motion and related
stochastic calculus under $G$-expectation, Stochastic Process. Appl. 118
(2008) 2223--2253.

\bibitem {Peng 1}S. Peng, Nonlinear expectations and stochastic calculus under
uncertainty, arXiv:1002.4546, 2010.

\bibitem {XSZ}J. Xu, H. Shang and B. Zhang, A Girsanov type theorem under
$G$-framework, Stoch. Anal. Appl. 29 (2011) 386-406.
\end{thebibliography}
\end{document}